\documentclass[a4paper, oneside]{article}

\usepackage{amsmath}
\usepackage{amsthm}
\usepackage{amscd}
\usepackage{amssymb}
\usepackage{array}
\usepackage{hyperref}
\usepackage{latexsym}
\usepackage{stmaryrd}
\usepackage[all]{xy}

\newtheorem{theorem}{Theorem}[section]

\newtheorem{lemma}[theorem]{Lemma}
\newtheorem{corollary}[theorem]{Corolllary}
\newtheorem{proposition}[theorem]{Proposition}

\begin{document}
\title{A characterization of weak (semi-)projectivity for commutative C*-algebras \\ \vspace{0.5cm} \normalsize{An appendix to "A characterization of semiprojectivity for commutative C*-algebras" by Adam P.W. S\o rensen and Hannes Thiel}}
\author{Dominic Enders}
\maketitle

\begin{abstract}
We show that the spectrum $X$ of a weakly semiprojective, commutative C*-algebra $C(X)$ is at most one dimensional. This completes the work of S\o rensen and Thiel on the characterization of weak (semi-)projectivity for commutative C*-algebras. 
\end{abstract}

\section{Introduction}
Let $X$ be a compact, metric space. Then the following statements hold:\\
\[\begin{tabular}{lcl}
$C(X)$ is projective in $\mathcal{S}_1$ & $\Leftrightarrow$ & $X$ is an AR and $dim(X)\leq 1$ \\
$C(X)$ is semiprojective in $\mathcal{S}_1$ & $\Leftrightarrow$ & $X$ is an ANR and $dim(X)\leq 1$
\end{tabular}\]
Here we denote by $\mathcal{S}_1$ the category of all unital, separable C*-algebras with unital *-homomorphisms. The first equivalence stated above is due to Chigogidze and Dranishnikov, see \cite{CD10}, while the second one was recently proved by S\o rensen and Thiel in \cite{ST11}. In this article we show that furthermore:
\[\begin{tabular}{lcl}
$C(X)$ is weakly projective in $\mathcal{S}_1$ & $\Leftrightarrow$ & $X$ is an AAR and $dim(X)\leq 1$ \\
$C(X)$ is weakly semiprojective in $\mathcal{S}_1$ & $\Leftrightarrow$ & $X$ is an AANR and $dim(X)\leq 1$
\end{tabular}\]
Note that the $"\Leftarrow"$-implications have already been proved in \cite[Corollary 6.16]{ST11}. Hence the only (non-trivial) part left to show is the dimension estimate for weakly (semi-)projective C*-algebras $C(X)$. This will be the content of this paper.\\
The idea of proof is the same as the one in \cite[Proposition 3.1]{ST11}: We show that if $C(X)$ was weakly semiprojective and $X$ an AANR of dimension $>1$, we could solve a lifting problem which is known to be unsolvable. We would like to point out that, as noted in \cite[Remark 3.3]{ST11}, the existence of an inclusion $\mathbb{D}^2\hookrightarrow X$ would be a sufficient, but not necessary condition to construct such a lifting problem. Note also that, since the closed two-dimensional disc $\mathbb{D}^2$ is an absolute retract, an embedding $\mathbb{D}^2\hookrightarrow X$ would admit a leftinverse $X\rightarrow\mathbb{D}^2$. We will show that for our purpose it is enough to have (possibly non-injective) maps $\mathbb{D}^2\rightarrow X$ which are leftinvertible in an extremely weak sense. The existence of such maps in the case of an AANR $X$ with dimension $>1$ is the crucial point in our argumentation and its proof makes up the greatest part of this paper. As an application of our main result, we illustrate by an example how one can generalize the results of \cite[Section 6]{ST11} to the setting of weakly (semi-)projective, commutative C*-algebras.\\

\section{Proof of the main result}
We refer the reader to section 2 of \cite{ST11} for definitions of weakly (semi-)projective C*-algebras, of AA(N)Rs, for further terminology, notations and everything else necessary.\\

Let us start off by a technical lemma concerning continuous self-maps of the closed, two-dimensional disc $\mathbb{D}^2=\{(x,y)\in\mathbb{R}^2:x^2+y^2\leq 1\}$.
\begin{lemma}\label{1}
Let $f:\mathbb{D}^2\rightarrow\mathbb{D}^2$ be continuous with the property that $|f(z)-z|<\frac{\sqrt{3}}{2}$ for all $z\in S^1\subset\mathbb{D}^2$. Then there exists a continuous map $g:\mathbb{D}^2\rightarrow\mathbb{D}^2$ such that $\|(g\circ f)-id_{\mathbb{D}^2}\|_{\infty}<1$.
\end{lemma}

\begin{proof}
A simple compactness argument shows that there is a $1-\frac{\sqrt{3}}{2}>t>0$ such that $|z|\geq 1-t$ implies $\left|\frac{z}{|z|}-f(z)\right|<\frac{\sqrt{3}}{2}$. Define $r_{max}, r_{min}:S^2\rightarrow [0,1]$ by 
\[r_{max}(z)=\max\{r\geq 0:rz\in f(S^1)\},\] resp. by \[r_{min}(z)=\min\{r\geq 0:rz\in f(S^1)\}.\] Notice that $r_{max}$ and $r_{min}$ are well defined, continuous and satisfy \[r_{max}\geq r_{min}> t\] by the assumption on $f$ and the choice of the parameter $t$.\\
Now let $r:\mathbb{D}^2\rightarrow [0,1]$ be the function
\begin{itemize}\item which equals $0$ on $\{se^{i\varphi}:s\leq r_{min}(e^{i\varphi})-t\}\cup\{se^{i\varphi}:s\geq r_{max}(e^{i\varphi})+t\}$, \item which equals $t$ on $\{se^{i\varphi}:r_{min}(e^{i\varphi})\leq s\leq r_{max}(e^{i\varphi})\}$ and \item such that for all $z\in S^1$ the maps $[0,1]\rightarrow[0,1]$, $s\mapsto r(sz)$ are linear on $[r_{min}(z)-t,r_{min}(z)]$ and $[r_{max}(z),r_{max}(z)+t]$.\end{itemize}
This map will be continuous. Finally, let $g:\mathbb{D}^2\rightarrow\mathbb{D}^2$ be given by
\[g(z)=r(z)z\]
Then $g$ is continuous with $\|g\|_{\infty}=t$ and we claim that $g$ satisfies $|(g\circ f)(z)-z|<1$ for all $z\in\mathbb{D}^2$. This is immediate for $|z|< 1-t$ since in this case $|(g\circ f)(z)-z|\leq t+|z|<1$. If $z=e^{i\varphi_0}$ we have $r(f(z))=t$ by construction of $r$ and $0\neq f(z)=|f(z)|e^{i\varphi}$ with $|\varphi-\varphi_0|<\frac{\pi}{3}$ by the assumption on $f$. Hence $(g\circ f)(z)=|(g\circ f)(z)|e^{i\varphi}$ with $|\varphi-\varphi_0|<\frac{\pi}{3}$ and $0<|(g\circ f)(z)|\leq t$. It follows that
\[|(g\circ f)(z)-z|\leq |te^{i\varphi_0}-z|+|(g\circ f)(z)-e^{i\varphi_0}|<(1-t)+t=1.\]
Now consider a fixed $z$ with $1-t\leq |z|< 1$ and write $z=|z|e^{i\varphi_0}$. The parameter $t$ was chosen in such a way that $f(z)=|f(z)|e^{i\varphi}$ with $|\varphi-\varphi_0|<\frac{\pi}{3}$. Hence $(g\circ f)(z)=|(g\circ f)(z)|e^{i\varphi}$ with $|\varphi-\varphi_0|<\frac{\pi}{3}$ and $|(g\circ f)(z)|\leq t$. Consequently, we estimate
\[|(g\circ f)(z)-z|\leq|te^{i\varphi_0}-z|+|(g\circ f)(z)-te^{i\varphi_0}|\leq (|z|-t)+t<1.\]
Using compactness of $X$ and continuity of $g\circ f$ we get the desired uniform estimate.
\end{proof}

The following is a refined version of an argument used in the proof of \cite[Proposition 3.1]{ST11}.

\begin{proposition}\label{2}
Let $X$ be a compact AANR with $dim(X)>1$. Then there exists a point $x_0\in X$ such that every neighbourhood of $x_0$ admits a topological embedding of $S^1$.
\end{proposition}

\begin{proof}
As shown in \cite[Theorem 3.4]{Gor99} there exists an ANR $Y$ such that $X$ is (homeomorphic to) an approximative retract of Y. This means we have $X\subseteq Y$ and for all $\epsilon>0$ there exists an $\epsilon$-retract $r_{\epsilon}:Y\rightarrow X$ (which again means that $d(x,r_{\epsilon}(x))<\epsilon$ for every $x\in X$).\\
We claim that for some $\epsilon>0$ we have $dim(r_\epsilon(Y))=dim(X)$. This follows from \cite[Corollary to Theorem V.9]{HW48}, since (following the terminology from \cite{HW48}) for every $\epsilon$-mapping $p$ from $r_{\epsilon}(Y)$ to some space $Z$, the composition 
\[\begin{xy}\xymatrix{X \ar[r]^{\subseteq} & Y \ar[r]^{r_{\epsilon}} & r_{\epsilon}(Y) \ar[r]^p & Z}\end{xy}\]
will be a $(3\epsilon)$-mapping on $X$.\\
Now fix some $\epsilon>0$ such that $dim(r_{\epsilon}(Y))=dim(X)>1$. 
By compactness we have $locdim(r_{\epsilon}(Y))=dim(r_{\epsilon}(Y))>1$. This means there exists a point $x_0\in r_{\epsilon}(Y)$ such that $dim(U)>1$ for every closed neighbourhood $U$ of $x_0$ in $r_{\epsilon}(Y)$. Since $Y$ is an ANR, it is a Peano continuum and so will be every continuous image of it. Hence $r_{\epsilon}(Y)$ is a Peano continuum of dimension $>1$ and so will be every closure of an open, connected neighbourhood $U$ of $x_0$ in $r_{\epsilon}(Y)$. Thus we may apply \cite[Proposition 3.1]{CD10} to obtain topological embeddings $S^1\hookrightarrow U$ for every such $U$. Now if $V$ is a closed neighbourhood of $x_0$ in $X$, $V\cap r_{\epsilon}(Y)$ is a closed neighbourhood of $x_0$ in $r_{\epsilon}(Y)$ and hence $S^1\hookrightarrow V\cap r_{\epsilon}(Y)\subseteq V$ as shown above.
\end{proof}

\begin{theorem}\label{3}
Let $X$ be a compact AANR with $dim(X)>1$. Then the following holds: There exist continuous maps $f:\mathbb{D}^2\rightarrow X$ and $g:X\rightarrow\mathbb{D}^2$ such that the diagram\\
\[\begin{xy}
\xymatrix{ & X \ar[dr]^g \\
\mathbb{D}^2 \ar[ur]^f \ar[rr]^{id} & & \mathbb{D}^2}
\end{xy}\]
commutes up to a constant strictly less then $1$, i.e. $\|g\circ f-id_{\mathbb{D}^2}\|_{\infty}<1$.
\end{theorem}

\begin{proof}
By \cite[Theorem 3.4]{Gor99} there exists an ANR $Y$ such that $X$ is an approximative retract of Y. This means we have $X\subseteq Y$ and for all $\epsilon>0$ there exists an $\epsilon$-retract $r_{\epsilon}:Y\rightarrow X$. Applying Proposition \ref{2}, we find some $x_0\in X$ such that every closed neighbourhood of $x_0$ in $X$ admits a topological embedding of $S^1$. Further there exists a closed neighbourhood $U$ of $x_0$ in $Y$ which is contractible in $Y$ since $Y$ is an ANR and hence locally contractible (see \cite[Section V.2]{Bor67}). Consequently, choose an embedding of $S^1$ in $U\cap X\subseteq X$. Moreover, the inclusion $S^1\subseteq U\cap X\subseteq U$ extends to a continuous map $f':\mathbb{D}^2\rightarrow Y$ since $U$ is contractible in $Y$.\\
Now choose a continuous extension $g':X\rightarrow\mathbb{D}^2$ of the canonical inclusiom $S^1\subset\mathbb{D}^2$. This is possible since $\mathbb{D}^2$ is an absolute retract. By compactness of $X$, $g'$ will uniformly continuous. So we find some $\delta>0$ such that $d(x,y)<\delta$ implies $|g'(x)-g'(y)|<\frac{\sqrt{3}}{2}$ for all $x,y\in X$. Let $r_{\delta}:Y\rightarrow X$ be a $\delta$-retract of $X$. We end up in the following situation\\
\[\begin{xy}
\xymatrix{S^1 \ar[r]^{\subseteq} \ar[d]_{\subseteq}&  X \ar[dl]_{g'} \\ \mathbb{D}^2 \ar[r]_{f'} & Y \ar[u]_{r_{\delta}}}
\end{xy}\]
where the upper left triangle commutes exactly and the outer square commutes up to $\delta$ on $S^1$. Setting $f:=r_{\delta}\circ f'$ we check that for $x\in S^1$\\
\[|(g'\circ f)(x)-x|\leq|g'(x)-x|+|g'((r_{\delta}\circ f')(x))-g'(x)|<0+\frac{\sqrt{3}}{2}=\frac{\sqrt{3}}{2}.\]
Applying Lemma \ref{1} gives some continuous $h:\mathbb{D}^2\rightarrow\mathbb{D}^2$ such that $\|h\circ(g'\circ f)-id_{\mathbb{D}^2}\|_{\infty}<1$. If we set $g:=h\circ g'$, we are done. 
\end{proof}

\begin{corollary}\label{4}
Let $C(X)$ be a unital, separable C*-algebra that is weakly semiprojective. Then $X$ is a compact AANR with $dim(X)\leq 1$.
\end{corollary}

\begin{proof}
Restriction to the category of commutative C*-algebras shows that $X$ is an AANR. Now assume that $dim(X)>1$ and let $f,g$ as in Theorem \ref{3} be given. We write $z=id_{\mathbb{D}^2}\in C(\mathbb{D}^2)$ for the canonical generator and regard $g$ as an element of $C(X)$. Then we have the induced $*$-homomorphism $f^*:C(X)\rightarrow C(\mathbb{D}^2)$ which satisfies
\[\|f^*(g)-z\|=\|g\circ f-z\|<1\]
by the choice of $f$ and $g$.\\
Now take a closer look at $C(\mathbb{D}^2)$: As noted in \cite[Proposition 3.2]{ST11} it is not weakly semiprojective. The proof shows that there is a C*-algebra $B$, an increasing chain of ideals $J_n$ in $B$ (we will write $J=\overline{\bigcup_n J_n}$) and a homomorphism $\varphi:C(\mathbb{D}^2)\rightarrow B/J$ which form an unsolvable lifting problem in the following way: For every $n\in\mathbb{N}$ and every homomorphism $\psi:C(\mathbb{D}^2)\rightarrow B/J_n$ we get $\|(\pi\circ\psi)(z)-\varphi(z)\|\geq 1$. This construction is due to Loring and based on the fact that \[dist(S,\{N+K\,|\,N,K\in B(H), N\,\text{normal}, K\,\text{compact}\})=1\]
where $S$ denotes the unilateral shift on a separable Hilbert space $H$.\\
Finally, choose $0<\epsilon<1-\|f^*(g)-z\|$. Since $C(X)$ is weakly semiprojective we can find an $n$ and a homomorphism $\varrho:C(X)\rightarrow B/J_n$ such that the inner square in the diagram
\[\begin{xy}
\xymatrix{C(X) \ar@{-->}[r]^{\varrho} \ar[d]^{f^*} & B/J_n \ar[d]^{\pi} \\ C(\mathbb{D}^2) \ar[r]^{\varphi} \ar@{-->}@/^/[u]^{\chi} & B/J}
\end{xy}\]
commutes up to $\epsilon$ on $g\in C(X)$. Since $\|g\|\leq 1$ we may use the universal property of $C(\mathbb{D}^2)$ to define a homomorphism $\chi:C(\mathbb{D}^2)\rightarrow C(X)$ via $\chi(z):=g$. But now we have constructed a lift $\psi:=\varrho\circ\chi:C(\mathbb{D}^2)\rightarrow B/J_n$ such that
\[\begin{tabular}{rl}
$\|(\pi\circ\psi)(z)-\varphi(z)\|$ & $=\|(\pi\circ\varrho)(g)-\varphi(z)\|$ \\
& $\leq\|(\varphi\circ f^*)(g)-\varphi(z)\|+\|(\pi\circ\varrho)(g)-(\varphi\circ f^*)(g)\|$ \\
& $\leq\|\varphi(f^*(g)-z)\|+\epsilon$\\
& $\leq\|f^*(g)-z\|+\epsilon<1.$
\end{tabular}\]
As noted above, such $\psi$ doesn't exist. This shows that our assumption was wrong and hence $dim(X)\leq 1$.
\end{proof}

\begin{corollary}\label{5}
Let $X$ be a compact, metric space. Then the following statements hold:
\[\begin{tabular}{lcl}
$C(X)$ is weakly projective in $\mathcal{S}_1$ & $\Leftrightarrow$ & $X$ is an AAR and $dim(X)\leq 1$ \\
$C(X)$ is weakly semiprojective in $\mathcal{S}_1$ & $\Leftrightarrow$ & $X$ is an AANR and $dim(X)\leq 1$
\end{tabular}\]
\end{corollary}

\begin{proof}
Corollary \ref{4} and \cite[Corollary 6.16]{ST11}.
\end{proof}

\section{An application}
With Corollary \ref{5} at hand, it is now an easy task to verify that the results of \cite[Section 6]{ST11} also hold for the case of weakly (semi-)projective, commutative C*-algebras. As an example we show the following 'weak' versions of Proposition 6.7, Proposition 6.8 and Corollary 6.9:

\begin{proposition}\label{6}
Let $X$ be a compact, metric space and $k\in\mathbb{N}$. If $C(X)\otimes M_k$ is weakly (semi-)projective, then $X$ is an AA(N)R.
\end{proposition}

\begin{proof}
Fix a metric $d$ on $X$ and suppose we are given a compact metric space $Y$ with an embedding $\iota:X\rightarrow Y$. We have to show that $\iota(X)$ is an approximative (neighbourhood) retract of $Y$ (compare \cite[Section I.3]{Gor99}).\\
So let $\epsilon>0$ and a decreasing sequence of closed neighbourhoods $U_n$ of $\iota(X)$ with $U_1=Y$ and $\bigcap_n U_n=\iota(X)$ be given. By compactnesss of $X$ we can find a finite subset $\{x_1,\cdots,x_n\}$ which is $\frac{\epsilon}{3}$-dense in $X$. For $i=1,\cdots,n$ define $g_i\in C(X,M_k)$ by
\[g_i(x)=\begin{pmatrix}d(x,x_i) \\ & \ddots \\ & & d(x,x_i)\end{pmatrix}.\]
By assumption we can find an $n$ and a homomorphism $\psi:C(X,M_k)\rightarrow C(U_n,M_k)$ with $\|((\iota_*)_k\circ\psi)(g_i)-g_i\|_{\infty}<\frac{\epsilon}{3}$ for all $i=1,\cdots,n$. Note that we may choose $n=1$ in the weakly projective case.\\
Exactly as in the proof of \cite[Proposition 6.6]{ST11} there is now a continuous map $\lambda: U_n\rightarrow X$ with
\[ev_x\circ((\iota_*)_k\circ\psi)=Ad_{u_{\iota(x)}}\circ ev_{(\lambda\circ\iota)(x)}\]
for suitable unitaries $u_{\iota(x)}\in M_k$ and all $x\in X$. We claim that $d(x,(\lambda\circ\iota)(x))<\epsilon$ for all $x\in X$, i.e. $\lambda$ is an $\epsilon$-retraction. Since $\epsilon$ was arbitrary, this will finish the proof. Now assume that there is some point $x\in X$ with $d(x,(\lambda\circ\iota)(x))\geq\epsilon$ and fix some $i$ with $d(x,x_i)<\frac{\epsilon}{3}$. Then
\[\begin{tabular}{rl}
$\|g_i((\lambda\circ\iota)(x))\|$ & $=\|\left(Ad_{u_{\iota(x)}}\circ ev_{(\lambda\circ\iota)(x)}\right)(g_i)\|=\|\left(ev_x\circ((\iota_*)_k\circ\psi)\right)(g_i)\|$ \\
& $\leq \|g_i(x)\|+\|((\iota_*)_k\circ\psi)(g_i)-g_i\|_{\infty}<\|g_i(x)\|+\frac{\epsilon}{3}$
\end{tabular}\]
contradicting the fact that
\[\begin{tabular}{rl}
$\|g_i((\lambda\circ\iota)(x))\|-\|g_i(x)\|$ & $=d((\lambda\circ\iota)(x),x_i)-d(x,x_i)$ \\
& $\geq d((\lambda\circ\iota)(x),x)-2d((x,x_i))\geq \epsilon -2\frac{\epsilon}{3}=\frac{\epsilon}{3}$.
\end{tabular}\]
\end{proof}

\begin{proposition}\label{7}
Let $X$ be a compact, metric space and $k\in\mathbb{N}$. If $C(X)\otimes M_k$ is weakly semiprojective, then $dim(X)\leq 1$.
\end{proposition}

\begin{proof}
By Proposition \ref{6} we know that $X$ is an AANR. We may now proceed as in the proof of corollary \ref{4} and use the assumption $dim(X)>1$ to construct a solution (up to a constant strictly less than one on $\{z\}$) to the lifting problem
\[\begin{xy}\xymatrix{& & B/J_n\otimes M_k \ar[d]^{(\pi)_k} \\ C(\mathbb{D}^2) \ar@{-->}[urr]^{\psi} \ar[rr]_{\varphi\oplus\cdots\oplus\varphi} & & B/J\otimes M_k}\end{xy}\]
Note that this lifting problem is still unsolvable since
\[dist(S^{\oplus k},\{N+K\,|\,N,K\in B(H^{\oplus k})\cong B(H)\otimes M_k, N\,\text{normal}, K\,\text{compact}\})=1\]
holds for all $k\in\mathbb{N}$. Hence we get the desired contradiction and conclude that $dim(X)\leq 1$.
\end{proof}

A combination of Corollary \ref{5}, Proposition \ref{6} and Proposition \ref{7} now gives

\begin{corollary}
Let $X$ be a compact, metric space and $k\in\mathbb{N}$. If $C(X)\otimes M_k$ is weakly projective, then so is $C(X)$. Analogously, if $C(X)\otimes M_k$ is weakly semiprojecive, then so is $C(X)$.
\end{corollary}

\noindent {\bfseries Acknowledgements.} I want to thank Hannes Thiel for his comments and help on the proof of Proposition \ref{2}.

\end{document}